\documentclass[psamsfonts,fceqn,leqno]{amsart}
\usepackage{mathrsfs,latexsym,amsfonts,amssymb,curves,epic}
\usepackage{color}

\newtheorem{theorem}{Theorem}[section]
\newtheorem{corollary}[theorem]{Corollary}
\newtheorem{proposition}[theorem]{Proposition}
\newtheorem{lemma}[theorem]{Lemma}
\newtheorem{question}[theorem]{Question}
\newtheorem{problem}[theorem]{Problem}
\theoremstyle{definition}

\newtheorem{remark}[theorem]{Remark}
\newtheorem{example}[theorem]{Example}

\def\N{{\mathbb N}}

\begin{document}

\title[The topological properties of $q$-spaces in free topological groups]
{The topological properties of $q$-spaces in free topological groups}

\author{Fucai Lin}
\address{(Fucai Lin): School of mathematics and statistics,
Minnan Normal University, Zhangzhou 363000, P. R. China}
\email{linfucai2008@aliyun.com}

\author{Chuan Liu}
\address{(Chuan Liu): Department of Mathematics,
Ohio University Zanesville Campus, Zanesville, OH 43701, USA}
\email{liuc1@ohio.edu}

\author{Shou Lin}
\address{(Shou Lin): Institute of Mathematics, Ningde Teachers' College, Ningde, Fujian
352100, P. R. China} \email{shoulin60@163.com}

\thanks{The first author is supported by the NSFC (Nos. 11201414, 11471153, 11571158), the
  Natural Science Foundation of Fujian Province (No. 2012J05013) of China and Training
  Programme Foundation for Excellent Youth Researching Talents of Fujian's Universities
  (JA13190) and the project of Abroad Fund of Minnan Normal University.}


\begin{abstract}
Given a Tychonoff space $X$, let $F(X)$ and $A(X)$ be respectively the free
topological group and the free Abelian topological group over $X$ in the sense
of Markov. In this paper, we provide some topological properties of $X$ whenever one of $F(X)$, $A(X)$, some finite level of $F(X)$ and some finite level of $A(X)$ is $q$-space (in particular, locally $\omega$-bounded spaces and $r$-spaces), which give some partial answers to a problem posed in \cite{T2003}.
\end{abstract}

\maketitle
\section{Introduction}
Let $F(X)$ and $A(X)$ be respectively the free topological group
and the free Abelian topological group over a Tychonoff space $X$
in the sense of Markov \cite{MA1945}. For every $n\in\mathbb{N}$, by $F_{n}(X)$ we denote the subspace of $F(X)$ that consists of all words of reduced length at most $n$ with respect to the free basis $X$. The subspace
$A_{n}(X)$ is defined similarly.
We always use $G(X)$ to denote $F(X)$ or $A(X)$, and $G_{n}(X)$ to $F_{n}(X)$ or $A_{n}(X)$ for each $n\in \mathbb{N}$. Therefore, any statement about $G(X)$ applies to $F(X)$ and $A(X)$, and $G_{n}(X)$ applies to $F(X)$ and $A(X)$.

One of the techniques of studying the topological structure of free topological groups is
to clarify the relations of subspaces $X$, $G(X)$, and $G_{n}(X)$, where $n\in\mathbb{N}$.
It is well known that only when $X$ is discrete, $G(X)$ can be first-countable. Therefore, $G(X)$ is first-countable if and only if $X$ is discrete \cite{G1962}. Similarly, the group $G(X)$ is locally compact if and only if $X$ is discrete \cite{D961}. More generally, P. Nickolas and M. Tkachenko proved that if $G(X)$ is almost metrizable, then $X$ is discrete \cite{T2003}.
Further, K. Yamada gave an characterization for a metrizable space $X$ such that some $G_{n}(X)$ is first-countable \cite{Y1998}.
Then P. Nickolas and M. Tkachenko showed that $F_{4}(X)$ is of pointwise countable type if and only if $X$ is either compact or discrete \cite{T2003}.
Since each space being of pointwise countable type is a $q$-space, P. Nickolas and M. Tkachenko posed the following problem in \cite{T2003}.

\begin{problem}\cite[Problem 5.2]{T2003}\label{pq}
Characterize the spaces $X$ such that $F_{2}(X)$ is a $q$-space. What about $F_{n}(X)$ for all $n\in\N$? The Abelian case?
\end{problem}

In this article, we shall give some partial answers to Problem~\ref{pq}. This article is organized as follows. Section 2 introduces the notation and terminology used throughout the paper. In Section 3, we characterize the spaces $X$ such
that $F(X)$ and $A(X)$ are $q$-spaces. Moreover, the $q$-subspaces of the groups $F(X)$ and $A(X)$ are studied. In Section 4, we show that $X$ must belong to some class of spaces if $F_{4}(X)$ is a $q$-space. We also show
that if $F_{2}(X)$ is a $q$-space then the set $X^{\prime}$ of all
non-isolated points of $X$ is bounded in $X$. In section 5,  we mainly prove that $F_{2}(X)$ is locally $\omega$-bounded if $X$ is homeomorphic to the topological sum of an $\omega$-bounded space and a
discrete space. In Section 6, some examples and questions are presented and posed respectively.

\maketitle
\section{Notation and Terminologies}
 In this section, we introduce the necessary notation and terminologies. Throughout this paper, all topological spaces are assumed to be
  Tychonoff, unless explicitly stated otherwise. For undefined notation
  and terminologies, refer to \cite{AT2008}, \cite{E1989} and \cite{Gr}.
  First of all, let $\N$ and $\omega_{1}$ denote the set of positive integers and the
  first uncountable order, respectively.

Let $X$ be a space and $x\in X$. If there exists a sequence $\{U_{n}: n\in\N\}$ of open neighborhoods of $x$ in $X$ satisfying the following {\bf ($\blacklozenge$)}, then we say that $x$ is a {\it $r$-point} (resp. $q$-point) in $X$ and the sequence $\{U_{n}: n\in\N\}$ is a {\it $r$-sequence} (resp. {\it $q$-sequence}).

  \medskip
{\bf ($\blacklozenge$)} \emph{For an arbitrary sequence $\{x_{n}\in U_{n}: n\in\N\}$, there exists a compact subset (resp. countably compact subset) $K$ of $X$ such that $\{x_{n}\in U_{n}: n\in\N\}\subset K.$}

  \medskip
A space $X$ is said to be a {\it $r$-space} (resp. {\it $q$-space}) if each point of $X$ is a $r$-point (resp. $q$-point) in $X$.
Recall that a space $X$ is {\it locally $\omega$-bounded} at point $x\in X$ if there exists an open neighborhood $U$ at $x$ in $X$ such that the closure of every countable subset of $\overline{U}$ is compact in $\overline{U}$. In particular, $X$ is said to be {\it locally $\omega$-bounded} if it is locally $\omega$-bounded at each point of $X$. A space $X$ is {\it $\omega$-bounded} \cite{AA2009} if the closure of every countable subset of $X$ is compact.
The following
  implications follow directly from definitions:

  \setlength{\unitlength}{1cm}
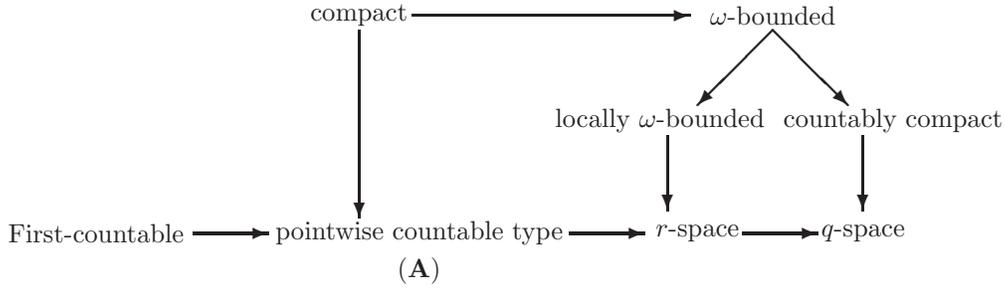
\begin{picture}(15,2.5)\thicklines
 \put(5, 1.7){\vector(0,-1){2.5}}
 \put(5.7, 1.9){\vector(1,0){3.7}}
 \put(5,1.9){\makebox(0,0){compact}}
 \put(10.5,1.9){\makebox(0,0){$\omega$-bounded}}
 \put(10.5, 1.7){\vector(1,-1){1}}
 \put(10.5, 1.7){\vector(-1,-1){1}}
 \put(12.1,0.5){\makebox(0,0){countably compact}}
 \put(9,0.5){\makebox(0,0){locally $\omega$-bounded}}
 \put(11.7,0.3){\vector(0,-1){1}}
 \put(9.1,0.3){\vector(0,-1){1}}
 \put(2.8,-1){\vector(1,0){1}}
 \put(1.5,-1){\makebox(0,0){First-countable}}
 \put(5.8,-1){\makebox(0,0){pointwise countable type}}
 \put(7.8,-1){\vector(1,0){1}}
 \put(9.5,-1){\makebox(0,0){$r$-space}}
 \put(10.1,-1){\vector(1,0){1}}
 \put(11.7,-1){\makebox(0,0){$q$-space}}
 \put(5.8,-1.5){\makebox(0,0){({\bf A})}}
\end{picture}

\vskip 2cm\setlength{\parindent}{0.5cm}

Note that none of the above implications can be reversed.

 \medskip
 Let $X$ be a topological space $X$ and $A$ be a subset of $X$.
  The \emph{closure} of $A$ in $X$ is denoted by $\overline{A}$. The subset $A$ is called
  \emph{bounded} if every continuous real-valued
  function $f$ defined on $A$ is bounded. If the closure of every bounded
  set in $X$ is compact, then $X$ is called a \emph{$\mu$-space}.
  The \emph{derived set}  of $X$ is denoted by
  $X'$. In addition, $X$ is called a
  \emph{$k$-space} provided that a subset $C\subseteq X$ is closed in $X$ if
 $C\cap K$ is closed in $K$ for each compact subset $K$ of $X$. In particular, $X$ is called a \emph{$k_{\omega}$-space} if there exists a family of countably many compact subsets $\{K_{n}: n\in\mathbb{N}\}$ in $X$ such that each subset $F$ of $X$ is closed in $X$ if $F\cap K_{n}$ is closed in $K_{n}$ for each $n\in\mathbb{N}$.

 \medskip
  Let $X$ be a non-empty space. Throughout this paper, $X^{-1}
  :=\{x^{-1}: x\in X\}$ and $-X: =\{-x: x\in X\}$, which are
 just two copies of $X$. Let $e$ be the neutral element of $F(X)$ (i.e., the empty
  word) and $0$ be that of $A(X)$. For every $n\in\N$ and an element
  $(x_{1}, x_{2}, \cdots, x_{n})$ of $(X\bigoplus X^{-1}\bigoplus\{e\})^{n}$
  we call $g=x_{1}x_{2}\cdots x_{n}$ a {\it form}. In the Abelian case, $x_{1}+x_{2}\cdots +x_{n}$ is also called a {\it form} for $(x_{1}, x_{2},
  \cdots, x_{n})\in(X\bigoplus -X\bigoplus\{0\})^{n}$. This word $g$ is
  called {\it reduced} if it does not contains $e$ or any pair of
  consecutive symbols of the form $xx^{-1}$ or $x^{-1}x$. It follows
  that if the word $g$ is reduced and non-empty, then it is different
  from the neutral element $e$ of $F(X)$. Similar
  assertions (with the obvious changes for commutativity) are valid for
  $A(X)$. For every $n\in\mathbb{N}$, let $i_n: (X\bigoplus X^{-1}
  \bigoplus\{e\})^{n} \to F_n(X)$ be the natural mapping defined by
  \[
  i_n(x_1, x_2, ... x_n)=
  x_1x_2...x_n
  \]
  for each $(x_1, x_2, ... x_n) \in (X\bigoplus X^{-1} \bigoplus\{e\})^{n}$.
  We also use the same symbol in the Abelian case, that is, in means
  the natural mapping from $(X\bigoplus -X\bigoplus\{0\})^{n}$ onto
  $A_{n}(X)$. Clearly, each $i_n$ is a continuous mapping.

\maketitle
\section{The characterizations of $q$-spaces in free topological groups}
In this section, we shall give a characterization for a space $X$ such that $G(X)$ is a $q$-space. Then, we show that $X$ contains a copy of $P$ if the free topological group $G(X)$ contains a copy of some $q$-space $P$.

We first show two lemmas and a proposition which play an important role in the proof of our main theorem in this section.

\begin{lemma}\label{open-neighborhood}
Let $X$ be a non-discrete space. Then, for each open neighborhood $U$ of $e$ in $A(X)$, we have $U\nsubseteq A_{k}(X)$ for each $k\in \mathbb{N}$.
\end{lemma}

\begin{proof}
Assume on the contrary that there exist an open neighborhood $U$ of $e$ in $A(X)$ and a $k_0\in\mathbb{N}$ such that $U\subseteq A_{k_0}(X)$. Since $A(X)$ is a topological group, there is an open neighborhood $V$ of $e$ in $A(X)$ with $V^{k_0+1}\subseteq U$. Since $X$ is not discrete, $A(X)$ is also not discrete. Take an arbitrary point $x\in V\setminus \{e\}$. Then we have $$x^{k_0+1}\in V^{k_0+1}\subseteq U\subseteq A_{k_0}(X).$$ Since $A(X)$ is Abelian, the length of $x^{k_0+1}$ is at least $k_0+1$. However, each element of $A_{k_0}(X)$ is at most $k_0$. A
contradiction occurs.
\end{proof}

It is well known that each $A(X)$ is a quotient group of $F(X)$ for each space $X$, hence $A(X)$ is an open image of $F(X)$. Then it follows from Lemma~\ref{open-neighborhood} that we also have the following lemma.

\begin{lemma}\label{open-neighborhood 1}
Let $X$ be a non-discrete space. Then, for each open neighborhood $U$ of $e$ in $F(X)$, we have $U\nsubseteq F_{k}(X)$ for each $k\in \mathbb{N}$.
\end{lemma}

\begin{proposition}\label{p1}
Suppose that $X$ is a space. If $Y$ is a $q$-subspace in $G(X)$. Then the family $\{(G_{n+1}\setminus G_{n})\cap Y: n\in\mathbb{N}\}$ is discrete in $Y$.
\end{proposition}

\begin{proof}
Assume on the contrary that the family
$\{(G_{n+1}\setminus G_{n})\cap Y: n\in\mathbb{N}\}$ is non-discrete in $Y$. Then there exits a point $g\in Y$ such that the family $\{(G_{n+1}\setminus G_{n})\cap Y: n\in\mathbb{N}\}$ is non-discrete at the point $g$ in $Y$. Since $Y$ is a $q$-space, there exists a sequence $\{U_{n}: n\in\mathbb{N}\}$ of open neighborhoods at $g$ in $Y$ satisfying ($\blacklozenge$). By the assumption, we can choose an increasing subsequence $\{k_{n}: n\in\mathbb{N}\}$ of $\mathbb{N}$ such that $U_{n}\cap (G_{k_{n}+1}\setminus G_{k_{n}})\neq\emptyset$ for each $n\in\mathbb{N}$, and then pick a point $g_{n}\in U_{n}\cap (G_{k_{n}+1}\setminus G_{k_{n}})$ for each $n\in\mathbb{N}$. By ($\blacklozenge$), the set $\{g_{n}: n\in\mathbb{N}\}$ is contained in some countably compact subset $K$ in $Y$. Obviously, $K$ is also a countably compact subset in $G(X)$, then it follows from \cite[Corollary 7.4.4]{AT2008} that $K$ is contained in some $G_{m}(X)$ for some $m\in\mathbb{N}$, which is a contradiction.
\end{proof}

By Lemmas ~\ref{open-neighborhood}, ~\ref{open-neighborhood 1} and Proposition~\ref{p1}, we obtain one of the main theorems in this section, which generalizes a result in \cite{T2003}.

\begin{theorem}\label{q-space}
Let $X$ be a space. If $G(X)$ is a $q$-space, then $X$ is discrete.
\end{theorem}

Moreover, by Proposition~\ref{p1}, we have the following result.

\begin{proposition}
Let $X$ be a non-discrete space. Then, for every sequence
$\{n_{i}: i\in \mathbb{N}\}$ of natural numbers, the set $\bigcup_{i\in\mathbb{N}}(G_{n_{i}}(X)\setminus G_{n_{i}-1}(X))$ is not a $q$-subspace.
\end{proposition}

\begin{proof}
By the proof of \cite[Corollary 3.4]{Y2002}, it easily check that $\{G_{n_{i}}(X)\setminus G_{n_{i}-1}(X): i\in\mathbb{N}\}$ is not discrete in $\bigcup_{i\in\mathbb{N}}(G_{n_{i}}(X)\setminus G_{n_{i}-1}(X))$. Therefore, it follows from Proposition~\ref{p1} that $\bigcup_{i\in\mathbb{N}}(G_{n_{i}}(X)\setminus G_{n_{i}-1}(X))$ is not a $q$-space.
\end{proof}

Next, we shall consider the question when $X$ contains a copy of $q$-space $P$ if $G(X)$ contains a copy of some $q$-space $P$. We first recall some concepts and show a technical lemma.

We say that a space $X$ is {\it densely self}-{\it embeddable}, if each non-empty open set
in $X$ contains a copy of $X$. A space $P$ is said to
be {\it prime} if for any spaces $X $ and $Y$, we have the following statement:

\medskip
\emph{If $X\times Y$ contains a copy of $P$,
then either $X$ or $Y$ contains a copy of $P$.}

\begin{lemma}\label{self-embeddable}
Assume that $Y$ is a densely self-embeddable, $q$-space and
$G(X)$ contains a copy of $Y$. Then, some $G_{n}(X)$ contains a copy of
$Y$.
\end{lemma}

\begin{proof}
Assume that $G(X)$ contains a copy of $Y$. Since $Y$ is a densely self-embeddable, it suffices to show that there exists a non-empty open set $U$ of $Y$ such that $U$ is contained in some $G_{n}(X)$. Pick an arbitrary point $y\in Y$. Then $y$ is a $q$-point in $Y$, and hence there exists a sequence $\{U_{n}: n\in\N\}$ of open neighborhoods at $y$ in $Y$ such that $\{U_{n}: n\in\N\}$ satisfies the condition ($\blacklozenge$). We claim that $U_{k}\subseteq G_{n}(X)$ for some $k, n\in\mathbb{N}$. Assume on the contrary that $U_{k}\not\subseteq G_{n}(X)$ for any $k, n\in\mathbb{N}$. Then it is easy to see that there exist a countably compact subset $K$ and a sequence $\{y_{n}: n\in\mathbb{N}\}$ in $Y$ such that $\{y_{n}: n\in\mathbb{N}\}\subseteq K$ and $y_{n}\not\in G_{n}(X)$ for each $n\in\mathbb{N}$. However, it follows from \cite[Corollary 7.4.4]{AT2008} that $K$ is contained in some $G_{n}(X)$, which is a contradiction. Therefore, $U_{k}\subseteq G_{n}(X)$ for some $k, n\in\mathbb{N}$, hence $G_{n}(X)$ contains a copy of
$Y$.
\end{proof}

Now we can show the second main theorem in this section, which generalizes a result in \cite{Y1995}. Indeed, K. Eda, H. Ohta and K. Yamada showed that, for each densely self-embeddable and prime space which is either compact
or first countable, $X$ contains a copy of $P$ iff  $F(X)$ contains a copy of $P$ iff $A(X)$ contains a copy of $P$. By the implications of the graph ({\bf A}), we see that each compact or first-countable space is a $q$-space.

\begin{theorem}\label{dense q}
Let $P$ be a densely self-embeddable, prime $q$-space. For a space $X$, the following are equivalent:
\begin{enumerate}
\item $X$ contains a copy of $P$.

\item $F(X)$ contains a copy of $P$.

\item $A(X)$ contains a copy of $P$.
\end{enumerate}
\end{theorem}

\begin{proof}
Clearly, it suffices to show that (2) $\Rightarrow$ (1) and (3) $\Rightarrow$ (1).

(2) $\Rightarrow$ (1). Assume that $F(X)$ contains a copy of $P$. By Lemma~\ref{self-embeddable}, some $F_{n}(X)$ contains a copy of
$P$. Without loss of generality, we may assume that $n$ is the least number such that $F_{n}(X)$ contains a copy of $P.$
Since $P$ is densely self-embeddable, $F_{n}(X)\setminus F_{n-1}(X)$ contains a copy of $P.$  Further, we may assume that $P$ is a subspace of $F_{n}(X)\setminus F_{n-1}(X)$. Then $(X\bigoplus X^{-1})^{n}$ contains a copy of
$P$ since $F_{n}(X)\setminus F_{n-1}(X)$ is homeomorphic
to a subspace of $(X\bigoplus X^{-1})^{n}$. Since $P$ is
prime and densely self-embeddable, $X$ contains a copy of $P$.

(3) $\Rightarrow$ (1). Similar to (2)$\Rightarrow$(1), one see that $A_{n}(X)\setminus A_{n-1}(X)$ contains a copy of
$P$ for some least number $n\in\mathbb{N}$. It is well known that $A_{n}(X)\setminus A_{n-1}(X)$ is homeomorphic
to a subspace of finite symmetric products $(X\bigoplus X^{-1})^{n}/ n!$ (see \cite{Y1995}). It follows from \cite[Proposition 2.1]{Y1995} that $X$ itself contains a copy of $P$.
\end{proof}

\maketitle
\section{The topological properties of $q$-spaces in $F_{n}(X)$}
In this section, we shall give some topological properties for a space $X$ whenever $G_{n}(X)$ is a $q$-space or $r$-space. First, we recall a theorem which was proved in \cite{Y1998}.

Let $X$ be a space, and let $\mathscr{U}_{X}$ be the finest uniformity on $X$ compatible with the topology of $X$. For each $U\in \mathscr{U}_{X}$, put $$O_{2}(U):=\{x^{\varepsilon}y^{-\varepsilon}: (x, y)\in U, \varepsilon=\pm 1\}.$$ In \cite{Y1998}, K. Yamada obtained the following theorem.

\begin{theorem}\cite{Y1998}\label{neighborhood base}
The family $\{O_{2}(U): U\in\mathscr{U}_{X}\}$ is a neighbourhood base at $e$ in $F_{2}(X)$.
\end{theorem}

Now we can give a topological property for a space $X$ when $F_{2}(X)$ is a $q$-space, which is similar to that in \cite[Lemma 2.1]{T2003}.

\begin{proposition}\label{p0}
If $F_{2}(X)$ is a $q$-space, then the derived set $X^{\prime}$ is bounded in $X$.
\end{proposition}

\begin{proof}
Since $F_{2}(X)$ is a $q$-space, there exists a sequence $\{U_{n}: n\in\N\}$ of open neighborhoods at $e$ in $F_{2}(X)$ such that $\{U_{n}: n\in\N\}$ satisfies ($\blacklozenge$). Suppose that $X^{\prime}$ is not bounded in $X$. Then there exists in $X$ a discrete family of open subsets $\{V_{n}: n\in\N\}$ each of which intersects $X^{\prime}$. Choose a point $x_{n}\in V_{n}\cap X^{\prime}$ for each $n\in\N$. Since each $x_{n}\in X^{\prime}$, we can find  $y_{n}\in V_{n}$ such that $x_{n}\neq y_{n}$ and $x_{n}^{-1}y_{n}\in U_{n}$ for each $n\in\N$. It is easy to see that $x_{n}^{-1}y_{n}\neq x_{m}^{-1}y_{m}$ for distinct $m, n\in\N$. Set $P:=\{x_{n}^{-1}y_{n}: n\in\N\}$.

{\bf Claim:} The set $P$ is closed and discrete in $F_{2}(X)$.

  \medskip
Indeed, it suffices to show that each point in $F_{2}(X)$ is not a cluster point of $P$.

  \medskip
{\bf Subclaim 1:} Each point in $X\cup X^{-1}$ is not a cluster point of $P$.

  \medskip
Obviously, all elements of $P$ have length 2. Since the set $X\cup X^{-1}$ is open in $F_{2}(X)$, the set $P$ has no cluster points in $X\cup X^{-1}$.

  \medskip
{\bf Subclaim 2:} The point $e$ is not a cluster point of $P$.

  \medskip
Since the family $\{V_{n}: n\in\N\}$ is discrete in $X$ and $x_{n}$ and $y_{n}$ are distinct points in $V_{n}$, it follows from \cite[8.1 (c)]{E1989} that the set $$U:=(X\times X)\setminus\{(x_{n}, y_{n}), (y_{n}, x_{n}): n\in\N\}$$ belongs to the finest uniformity on $X$. Then it follows from Theorem~\ref{neighborhood base} that the set $O_{2}(U)$ is an open neighborhood at $e$ in $F_{2}(X)$. Obviously, $O_{2}(U)\cap P=\emptyset$. Therefore, $e$ is not a cluster point of $P$.

  \medskip
{\bf Subclaim 3:} Each point in $F_{2}(X)\setminus F_{1}(X)$ is not a cluster point of $P$.

  \medskip
Clearly, the natural mapping $i_{2}: (X\bigoplus X^{-1}\bigoplus\{e\})^{2}\rightarrow F_{2}(X)$ is a local homeomorphism from $(X\bigoplus X^{-1}\bigoplus\{e\})^{2}$ to $F_{2}(X)$ at each point of $(X\bigoplus X^{-1}\bigoplus\{e\})^{2}\setminus i_{2}^{\leftarrow}(e)$. Moreover, the set $L:=\{(x_{n}^{-1}, y_{n}): n\in\N\}$ is closed and discrete in $(X\bigoplus X^{-1}\bigoplus\{e\})^{2}$, which shows that the set $P=j_{2}(L)$ has no cluster points in $F_{2}(X)\setminus F_{1}(X)$.

Therefore, the claim is verified.

However, $F_{2}(X)$ is a $q$-space, then $P$ has a cluster point in $F_{2}(X)$ since each $x_{n}^{-1}y_{n}\in U_{n}$, which is a contradiction.

Therefore, the derived set $X^{\prime}$ is bounded in $X$.
\end{proof}

By Proposition~\ref{p0}, the following two corollaries are obvious.

\begin{corollary}
If $X$ is a homogeneous space and $F_{2}(X)$ is a $q$-space, then $X$ is pseudocompact or discrete.
\end{corollary}

\begin{corollary}
If $F_{2}(X)$ is a $q$-space and $X$ is a $\mu$-space, then $X^{\prime}$ is compact.
\end{corollary}

Next, we shall consider the question that $X$ belongs to what kind of classes of spaces if $F_{4}(X)$ is a $q$-space. We first recall two concepts.

A space is called a {\it cf}-{\it space} if each compact subset of it is finite. Recall that a subspace $Y$ of a space $X$ is
  said to be {\it P-embedded in $X$} if each continuous pseudometric
  on $Y$ admits a continuous extensions over $X$.

\begin{theorem}\label{p5}
Let $X$ be a space. If $F_{4}(X)$ is a $q$-space, then $X$ is either pseudocompact or a $cf$-space.
\end{theorem}

\begin{proof}
Suppose $X$ is not a $cf$-space. Then there exists an infinite compact subset $C$ in $X$.
Next we shall show that $X$ is pseudocompact.

Assume on the contrary that we can find a countably infinite discrete family of open sets $\{U_{n}: n\in\N\}$ in $X$. It easily check that the family of $\{\overline{U_{n}}: n\in\N\}$ is also discrete in $X$, and hence $\bigcup_{n\in\N}\overline{U_{n}}$ is closed in $X$. Since the set $C$ is compact, it can intersect at most a finite number of the sets $\overline{U_{n}}$. Without loss of generality, we may assume that $C\cap \bigcup_{n\in\N}\overline{U_{n}}=\emptyset$. Now the family $\{C\}\cup\{U_{n}: n\in\N\}$ is discrete in $X$.

Since $C$ is an infinite compact subset in $X$, there exists a non-isolated point $x$ in $C$. For each $n\in\N$, pick $y_{n}\in U_{n}$, and then put $C_{n}:=y_{n}^{-1}x^{-1}C y_{n}$. Let $$Y:=C\cup\{y_{n}: n\in\N\}\ \mbox{and}\ Z:=\bigcup_{n\in\N} C_{n}.$$

Obviously, $Y$ is closed, $\sigma$-compact, non-discrete and $P$-embedded in $X$. By \cite{S2000}, the subgroup $F(Y, X)$ of $F(X)$ generated  by $Y$ is naturally topologically isomorphic to the free topological group $F(Y)$. Moreover, since $Y$ is a $k_{\omega}$-space, it follows from \cite{AT2008} that the free group $F(Y)$ is also a $k_{\omega}$-space, hence $F_{4}(Y)$ is also a $k_{\omega}$-space. Next, we claim that $Z$ is closed in $F_{4}(Y)$.

Indeed, for each $n\in \mathbb{N}$, let $$K_{n}:=C\cup C^{-1}\cup\{y_{i}, y_{i}^{-1}: i\leq n\}\cup\{e\}$$and $$X_{n}:=\overbrace{K_{n}\cdot\ldots \cdot K_{n}}^{n}$$Then it follows from the proof of \cite[Theorem 7.4.1]{AT2008} that the topology of $F(Y)$ is determined by the family of compact subsets $\{X_{n}: n\in\mathbb{N}\}$. Hence the topology of $F_{4}(Y)$ is determined by the family $\{X_{n}\cap F_{4}(Y): n\in\mathbb{N}\}$. Then we have $Z\cap X_{n}\cap F_{4}(Y)=\bigcup_{i=1}^{n}C_{i}$ for each $n\in\mathbb{N}$, which shows that $Z$ is closed in $F_{4}(Y)$. Therefore, $Z$ is a $q$-space. Moreover, ones see that the topology of $Z$ is determined by the family $\{C_{n}: n\in\mathbb{N}\}$, which shows that $Z$ is not a $q$-space since any compact subset of $Z$ lies in a finite union $\bigcup_{i\leq n}C_{i}$ for some $n\in \mathbb{N}$. Hence $e$ is not a $q$-point in $F_{4}(X)$, which is a contradiction.
\end{proof}

\begin{remark}
Let $Z=X\bigoplus Y$, where $X$ is an infinite compact space and $Y$ is an infinite discrete space. Then $F_{4}(Z)$ is not a $q$-space by Theorem~\ref{p5}. By \cite{Y1998}, ones see that ``$F_{4}(X)$'' in Theorem~\ref{p5} cannot be replaced by ``$F_{3}(X)$''. Indeed, if $X$ is a compact metrizable space, then $F_{3}(Z)$ is metrizable \cite{Y1998}. However, $Z$ is neither pseudocompact nor a $cf$-space.
\end{remark}

It is natural to consider what conditions on a space $X$ can guarantee $X$ to be countably compact or discrete if $F_{4}(X)$ is a $q$-space. We first recall two concepts.

A topological space is said to be {\it collectionwise Hausdorff} if given any closed discrete collection of points in the topological space, there are pairwise disjoint open sets containing the points. A topological space $X$ is {\it weakly}-$k$ if and only if a subset $F$ of $X$ is closed
in $X$ if $F\cap C$ is finite for every compact subset $C$ in $X$. Obviously, we have the following two facts:

(a) Each $k$-space is a weakly-$k$ space.

(b) Each weakly-$k$ $cf$-space is discrete.

However, there exists a weakly-$k$ space which
is not a $k$-space, see \cite{R1976}.

Obviously, each collectionwise Hausdorff pseudocompact space is countably compact. Hence, we have the following corollary.

\begin{corollary}\label{cooo}
Let $X$ be a collectionwise Hausdorff space. If $F_{4}(X)$ is a $q$-space, then $X$ is either countably compact or a $cf$-space.
\end{corollary}

Further, we have the following Theorem~\ref{tco}, which is a more stronger form of Corollary~\ref{cooo}.

\begin{theorem}\label{tco}
Let $X$ be a collectionwise Hausdorff, weakly-$k$ and non-discrete space. If $F_{4}(X)$ is a $q$-space, then $X^{2}$ is countably compact.
\end{theorem}

\begin{proof}
Since $X$ is a non-discrete weakly-$k$-space, we see that $X$ is not a $cf$-space, hence $X$ is pseudocompact by Theorem~\ref{p5}. Then $X$ is countably compact since it is collectionwise Hausdorff. Therefore, it follows from \cite{R1976} that $X^{2}$ is countably compact.
\end{proof}

If we replace ``$q$-space'' with ``$r$-space'' in Theorem~\ref{p5}, we can obtain a more stronger result.

\begin{corollary}\label{p2}
Let $X$ be a space. If $F_{4}(X)$ is a $r$-space, then $X$ is either pseudocompact or discrete.
\end{corollary}

\begin{proof}
Suppose that $X$ is a $cf$-space, then $X$ is discrete since $X$ is a $r$-space. Therefore, the corollary holds by Theorem~\ref{p5}.
\end{proof}

By Corollary~\ref{p2}, we have the following theorem.

\begin{theorem}\label{t4}
The following conditions are equivalent for a $\mu$-space $X$:
\begin{enumerate}
\item $F_{4}(X)$ is a $r$-space.
\item $F_{4}(X)$ is of pointwise countable type.
\item $F_{4}(X)$ is locally compact.
\item $F_{n}(X)$ is of pointwise countable type for each $n\in\N$.
\item $F_{n}(X)$ is locally compact for each $n\in\N$.
\item The space $X$ is either compact or discrete.
\end{enumerate}
\end{theorem}

\begin{proof}
For any space $X$, the equivalence of items (2), (3), (4) and (5) of the theorem were established in \cite{T2003}. Obviously, we have (5) $\Rightarrow$ (1). By Corollary~\ref{p2}, $X$ is pseudocompact or discrete. If $X$ is pseudocompact, it follows that $X$ is compact since it is a $\mu$-space. Hence (1) $\Rightarrow$ (2).
\end{proof}

\maketitle
\section{The topological properties of $\omega$-bounded spaces in $F_{2}(X)$}
In this section, we shall give some topological properties for a space $X$ when $G_{2}(X)$ is an $\omega$-bounded space. First, we give some topological properties of an $\omega$-bounded space.

Obviously, each compact space is $\omega$-bounded, and each $\omega$-bounded space is a countably compact $r$-space. Moreover, it is easy to see that $[0, \omega_{1})$ is an $\omega$-bounded and non-compact space. It is well known that there exists two countably compact spaces such that their product is not countably compact. However, it easily checked that the product of a family of $\omega$-bounded spaces is again $\omega$-bounded.

The following two propositions are obvious.

\begin{proposition}\label{p3}
Let $f: X\rightarrow Y$ be a continuous mapping. If $X$ is $\omega$-bounded, then $Y$ is also $\omega$-bounded.
\end{proposition}

By \cite[Propositions 3.72 and 3.7.6]{E1989}, it is easy to show the following proposition.

\begin{proposition}\label{p333}
Let $f: X\rightarrow Y$ be a continuous perfect mapping. If $Y$ is $\omega$-bounded (resp. locally $\omega$-bounded), then $X$ is also $\omega$-bounded (locally $\omega$-bounded).
\end{proposition}

\begin{proposition}\label{p55}
Let $X=Y\bigoplus D$ be the topological sum of a space $Y$ and a discrete space $D$. Then we have the following statements.

(a) $F_{2}(X)$ is a $r$-space iff $F_{2}(Y)$ is a $r$-space.

(b) $F_{2}(X)$ is a locally $\omega$-bounded space iff $F_{2}(Y)$ is a locally $\omega$-bounded space.

(c) For every $n\in\N$, $A_{n}(X)$ is a $q$-space (resp. $r$-space, locally $\omega$-bounded) iff $A_{n}(Y)$ is a $q$-space (resp. $r$-space, locally $\omega$-bounded).
\end{proposition}

\begin{proof}
($a$) Since $Y$ is a retract of $X$, the closed subgroup $F(Y, X)$ of $F(X)$ is topologically isomorphic to the group $F(Y)$. Therefore, $F_{2}(Y)$ can be identified with the closed subspace $F_{2}(Y, X):=F(Y, X)\cap F_{2}(X)$ of $F_{2}(X)$. Thus $F_{2}(Y)$ is a $r$-space if $F_{2}(X)$ is a $r$-space.

Conversely, suppose that $F_{2}(Y)$ is a $r$-space. Since $Y$ is closed in $X$, it is easy to see that $X$ is a $r$-space. Then $(X\bigoplus X^{-1}\bigoplus\{e\})^{2}$ is a $r$-space. Further, since the natural mapping $i_{2}: (X\bigoplus X^{-1}\bigoplus\{e\})^{2}\rightarrow F_{2}(X)$ is a local homeomorphism from $(X\bigoplus X^{-1}\bigoplus\{e\})^{2}$ to $F_{2}(X)$ at each point of $(X\bigoplus X^{-1}\bigoplus\{e\})^{2}\setminus i_{2}^{\leftarrow}(e)$, each point of $F_{2}(X)\setminus \{e\}$ is a $r$-point. It remains to show that $e$ is a $r$-point in $F_{2}(X)$.

Since $e$ is a $r$-point in $F_{2}(Y)$, there exists a sequence of open neighborhoods $\{U_{n}: n\in\N\}$ of $e$ in $F_{2}(Y)$ such that $\{U_{n}: n\in\N\}$ is a $r$-sequence. For every $n\in\N$, it follows from Theorem~\ref{neighborhood base} that there exists a $W_{n}\in\mathscr{U}_{Y}$ such that $O_{2}(W_{n})\subset U_{n}$. Put $W^{\ast}_{n}:=W_{n}\cup\{(d, d): d\in D\}$ for each $n\in\N$. Clearly, each $W^{\ast}_{n}$ belongs to the finest uniformity of $X$, and $O_{2}(W_{n})=O_{2}(W^{\ast}_{n})$. Since $F(Y)$ is topologically isomorphic to the closed subgroup $F(Y, X)$ of $F(X)$, the sequence $\{O_{2}(W^{\ast}_{n}): n\in\N\}$ is a $r$-sequence at $e$ in $F_{2}(X)$.

($b$) The proof is quite similar to ($a$), so we omit it.

($c$) It is well known that $A(X)$ is topologically isomorphic in the natural way with $A(Y)\times A(D)$. If we identify $A(X)$ and $A(Y)\times A(D)$ under this isomorphism, then we have $A_{n}(X)\subseteq A_{n}(Y)\times A_{n}(D)$ and $A_{n}(X)$ is closed in $A_{n}(Y)\times A_{n}(D)$ for each $n\in\N$. Since $A(D)$ is discrete, each $A_{n}(D)$ is discrete, and hence it is easy to see that (c) holds.
\end{proof}

\begin{proposition}
If  $A_{2}(X)$ is locally $\omega$-bounded, then $F_{2}(X)$ is also locally $\omega$-bounded
\end{proposition}

\begin{proof}
 Let $\varphi: F(X)\rightarrow A(X)$ be the canonical homomorphism from $F(X)$ onto $A(X)$. Then the restriction $\varphi_{2}:=\varphi\upharpoonright F_{2}(X)$ is a perfect
mapping from $F_{2}(X)$ onto $A_{2}(X)$ by \cite[Lemma 2.10]{T2003}. Hence $F_{2}(X)$ is locally $\omega$-bounded by Proposition~\ref{p333}.
\end{proof}

By Proposition~\ref{p55}, we obtain the main result in this section.

\begin{theorem}\label{t222}
If $X$ is homeomorphic to the topological sum of an $\omega$-bounded space and a discrete space, then $F_{2}(X)$ is locally $\omega$-bounded.
\end{theorem}

\maketitle
\section{Examples and Questions}
First, some examples are presented to show the applications of our results.

\begin{example}\label{e0}
There exists a first-countable space $X$ such that $F_{4}(X)$ is an $\omega$-bounded space (in particular, it is a $r$-space) and $F_{2}(X)$ is not of pointwise countable type.
\end{example}

\begin{proof}
Let $X:=[0, \omega_{1})$ be endowed with the order topology. Then $X$ is an $\omega$-bounded and non-compact space. By Proposition~\ref{p3}, the subspace $F_{4}(X)$ is an $\omega$-bounded space. Suppose that $F_{2}(X)$ is of pointwise countable type. Then the derived set $X^{\prime}$ is compact in $X$ by \cite[Lemma 2.1]{T2003}, and hence $X$ is paracompact, which is a contradiction.
\end{proof}

\begin{remark}
 (1) By Example~\ref{e0}, we can not omit the condition ``$\mu$-space'' in Theorem~\ref{t4}.

(2) By Example~\ref{e0}, we know that $X$ may not be compact in Theorem~\ref{p5}.
\end{remark}

\begin{example}
There exists a space $X$ such that $F_{2}(X)$ is a $r$-space and $F_{4}(X)$ is not a $r$-space.
\end{example}

\begin{proof}
Let $X:=[0, \omega_{1})\bigoplus D$, where $[0, \omega_{1})$ is endowed with the order topology and $D$ is an uncountable set with a discrete topology. By Theorem~\ref{t222}, $F_{2}([0, \omega_{1})\bigoplus D)$ is a $r$-space. However, it follows from Corollary~\ref{p2} that $F_{4}([0, \omega_{1})\bigoplus D)$ is not a $r$-space since $[0, \omega_{1})\bigoplus D$ is not pseudocompact.
\end{proof}

In \cite{T2003}, P. Nickolas and M. Tkachenko proved that $F_{2}(X)$ is locally compact iff $A_{2}(X)$ is locally compact iff $X$ is homeomorphic to the topological sum of a compact space and a discrete space. From  Theorem~\ref{t222}, it is natural to pose the following question.

\begin{question}
Let $F_{2}(X)$ be locally $\omega$-bounded. Is $X$ homeomorphic to the topological sum of an $\omega$-bounded space and a discrete space?
\end{question}

Further, we have the following question.

\begin{question}
If $F_{2}(X)$ is a non-discrete $r$-space, is $X$ homeomorphic to the topological sum of an $\omega$-bounded space and a discrete space?
\end{question}

In \cite{T2003}, P. Nickolas and M. Tkachenko proved that $F_{2}(X)$ is locally compact iff $F_{3}(X)$ is locally compact for any space. Therefore, it is natural to pose the following question.

\begin{question}
If $F_{2}(X)$ is a $r$-space (resp. $q$-space), is $F_{3}(X)$ also a $r$-space ($q$-space)?
\end{question}

By Theorem~\ref{p5}, we do not know the answer to the following question.

\begin{question}
Let $X$ be a space. If $F_{4}(X)$ is a $q$-space, is $X$ either countably compact or a $cf$-space?
\end{question}

\end{document}